\documentclass[reqno]{amsart}
\usepackage[margin=1.00in]{geometry}
\usepackage{graphicx}
\begin{document}
\numberwithin{equation}{section}
\newtheorem{theorem}{Theorem}[section]
\newtheorem{lemma}[theorem]{Lemma}
\newtheorem{definition}[theorem]{Definition}
\newtheorem{remark}[theorem]{Remark}
\allowdisplaybreaks

\title[]{Some sufficient conditions on impulsive and initial value fractional order functional differential equations}
\author[G. R. Gautam , A. Dwivedi, G. Rani, \hfil \hfilneg] {Ganga Ram Gautam, Arpit Dwivedi, Gunjan Rani \\
  DST-Centre for Interdisciplinary Mathematical Sciences, Institute of Science, \\ Banaras Hindu University, Varanasi-221005
  India.}
\address{Ganga Ram Gautam, Arpit Dwivedi, Gunjan Rani: DST-Centre for Interdisciplinary Mathematical Sciences, Institute of Science, \\ Banaras Hindu University, Varanasi-221005 India.}
\email{gangacims@bhu.ac.in, arpit@bhu.ac.in, gunjan0806@bhu.ac.in}

\begin{abstract}
In this  paper, we will develop a definition of mild solution for
impulsive fractional differential equation of order $\alpha\in
(1,2)$ with the help of solution operator and study the existence
results of mild solution for impulsive fractional differential
equations with state dependent delay by using fixed point theorems.
Finally, we present an example of partial fractional derivative to
illustrate the existence and uniqueness result.
\end{abstract}

\maketitle

\section{Introduction}
For the past few decades,  the study of the theory of fractional
differential equations in infinite dimensional spaces have become an
important role due to their numerous application in several fields
of science and  engineering. One can see the monographs
\cite{sgaa,f-2,ksbr,ipfd,aaha,vlsl} for more detail of the
fractional calculus. On the other side differential equations with
impulsive effects are paid attention by many researchers because the
model processes which are subjected to abrupt changes are not
described by ordinary differential equations, so such type equations
can be modeled in term of impulses. The most important applications
of these equations are  in the ecology, mechanics, electrical,
medicine and biology. Fractional differential equations with
state-dependent delay normally arise in the remote control, implicit
equations like Wheeler-Feynman equations, structured populations
model which involve threshold phenomena etc. For more details one
can see the papers \cite{k1,mbfb,rpbd,k2} and references therein.

Feckan et al.\cite{mfyz} introduced a correct formula of solutions
for a impulsive Cauchy problem with Caputo fractional derivative and
 some sufficient conditions for existence of the solutions are
established by applying fixed point methods. Wang  et al.
\cite{jwmf} studied the  existence of $PC$-mild solutions for Cauchy
problems and nonlocal problems for impulsive fractional evolution
equations involving Caputo fractional derivative by utilizing the
theory of operators semigroup, probability density functions via
impulsive conditions. Dabas and Chauhan  \cite{jdace} obtained
 the existence, uniqueness and continuous dependence of mild solution for an impulsive neutral fractional
 order differential equation of order $\alpha \in (0,1)$ with infinite delay by using the fixed point technique and solution operator on a complex Banach space.
Wang \cite{jwxl} extended the problem consider in \cite{mfyz} for
order $\alpha \in (1,2).$ Shu et al. \cite{xsqw} determined the
definition of mild solution for fractional differential equations
with nonlocal conditions of order $\alpha \in (1,2)$ without
impulse. However, it should be pointed out that no work has been
reported in the existing literature regarding the existence of mild
 solution for impulsive fractional differential equation of order
$\alpha \in (1,2).$

Functional differential equations originate in several branches of
engineering, applied mathematics and science. Recently, fractional
functional differential equations with state dependent delay seems
frequently in many fields for modeling of equations such as panorama
of natural phenomena and porous media. See for more details of
relevant developments theory in the cited papers
\cite{acjd,acjdb,f-1,f-3,f-5,jdgr,acjdm,xfrh,mbfb,rpbd,sbmb,mbsl,jpmm}.

Inspired above mention works, in this paper 
we develop the definition of mild solution for impulsive fractional
differential equation of order  $\alpha \in (1,2)$ and show the
existence results for impulsive differential equation with state
dependent delay of the form:
\begin{eqnarray}
\label{p}^CD_t^\alpha u(t)&=&Au(t)+f(t,u_{\rho(t,u_t)}),\;t\in J=[0,T],\;t\neq t_k,\\
\label{c}u(t)&=&\phi(t),\;u'(t)=\varphi(t),t\in [-d,0],\\
\label{ci}\Delta u(t_k)&=&I_k(u(t_k^-)),\;\Delta u'(t_k)=Q_k(u(t_k^-)),\quad k=1,2,...m,
\end{eqnarray}
where $^CD^\alpha _t$ is the Caputo's fractional
derivative of order $\alpha \in (1,2)$ and $J$ is operational interval. $A:D(A)\subset X\rightarrow X $
is the sectorial operator defined on a complex Banach space $X$. The
 functions $f:J\times PC_0\rightarrow X ,\;\rho:J\times PC_0\to[-d,T]$ and
$\phi,\varphi,\in PC_0$ are given and satisfies some assumptions,
where $PC_0$ introduced in section $2$. The
history function $u_t:[-d,0]\rightarrow X$ is
 defined by $ u_t(\theta)=u(t+\theta),\; \theta \in [-d, 0]$
belongs to $PC_0$. Here $0\leq
t_0<t_1<...<t_m<t_{m+1}\leq T<\infty,$ the functions $I_k,Q_k\in C(X,X),\;k=1,
2, ...m,$ are bounded. We have
 $\triangle u(t_k)=u(t_k^+)-u(t_k^-)$ where $u(t_k^+)$ and $u(t_k^-)$
represent the right and left- hand limits of $u(t)\, \mbox{at}\,
t=t_k$ respectively, also we take $u(t_i^-)=u(t_i)$ and $\triangle u'(t_k)=u'(t_k^+)-u'(t_k^-)$ where $u'(t_k^+)$ and $u'(t_k^-)$
represent the right and left- hand limits of $u(t)\, \mbox{at}\,
t=t_k,$ also we take $u'(t_i^-)=u'(t_i)$ respectively.

For further details, this work has four sections, second section
provides some basic definitions, preliminaries, theorems and lemmas.
Third section is equipped with main results for the considered
problem (\ref{p})-(\ref{ci}) and  fourth section has  an example.

\section{Preliminaries and back ground martials}
Let $(X,\|\cdot\|_X)$ be a complex Banach space of functions with the norm $\|u\|_{X}=\sup_{t\in J} \{|u(t)|:u\in X\}$ and $L(X)$ denotes the Banach space of bounded linear operators from $X$ into $X$ equipped with norm is denoted by $\|.\|_{L(X)}.$

Let $PC_0=C([-d,0], X)$ (with $[-d,0]\subset \mathbb{R}$) denotes
the space formed by all the continuous functions defined from
$[-d,0]$ into $X,$ endowed with the norm  $\|u(t)\|_{C([-d,0],
X)}=\sup_{t\in [-d,0]}\|u(t)\|_X.$

 To study the impulsive conditions, we consider
$PC^2_t=PC([-d,t];X),\; 0< t\leq T<\infty$ be a Banach space of all
 functions $u:[-d,T]\to X,$ which have $2-$times continuously
differentiable on $[0,T]$ except for a finite number of points
$t_i\in (0,T),\;i=1,2,\dots,\mathbb{N},$ at which $u'(t_i^+)$ and
$u'(t_i^-)=u'(t_i)$ exists and endowed with the norm
$\|u\|_{PC_t^2}=\sup_{t\in [-d,T]}\sum_{j=0}^{2}\left\{\|u^j(t)\|_X,
u\in PC_t^2\right\}.$

\begin{definition}\label{cpto}\rm
Caputo's derivative of order $\alpha>0$ with lower limit $a,$ for a function $f:
[a,\infty)\to \mathbb{R}$ such that $f\in C^n([a,\infty),X)$ is defined as
\begin{eqnarray*}
_aD^\alpha_tf(t)={1\over
\Gamma(n-\alpha)}\int_a^t(t-s)^{n-\alpha-1}f^{(n)}(s)ds=_aJ_t^{n-\alpha}f^{(n)}(t),
\end{eqnarray*}
where $a\geq0,\;n-1<\alpha<n,\;n\in \mathbb{N}.$
\end{definition}

\begin{definition}\label{remn}\rm
The Riemann-Liouville fractional integral operator of order $\alpha
> 0$ with lower limit $a,$ for a continuous function $f:
[a,\infty)\to \mathbb{R}$ is defined by
\begin{eqnarray*} \label{e2.1}
_aJ_t^0f(t)=f(t),\;_aJ_t^{\alpha}f(t)={1\over
\Gamma(\alpha)}\int_a^t(t-s)^{\alpha-1}f(s)ds,\quad \;t>0,
\end{eqnarray*}
where $a\geq0$ and $\Gamma(\cdot)$ is the Euler gamma function.
\end{definition}

\begin{definition}(\cite{xsqw})\label{def2.3} \rm
Let $A:D(A)\subseteq X\to X$ be a densely defined, closed and linear
operator in $X$. $A$ is said to be sectorial of the type
$(M,\theta,\alpha,\mu)$ if there exist  $\mu \in \mathbb{R},\;
\theta\in({\pi\over2},\pi),\; M>0,$ such that  the
$\alpha$-resolvent of $A$ exists outside the sector and following
two conditions are satisfied:
\begin{itemize}
\item[(1)] $\mu+S_{\theta}=\{\mu+\lambda^{\alpha}:\lambda\in \mathbb{C},\; |Arg(-\lambda^{\alpha})|<\theta\},$
 \item[(2)]$\|(\lambda^{\alpha}I-A)^{-1}\|_{L(X)}\leq{M\over{|\lambda^{\alpha}-\mu|}}, \;\lambda \notin \mu+S_{\theta},$
\end{itemize}
where X is the complex Banach space with norm denoted
$\|.\|_{X}$.
 \end{definition}

\begin{definition}\label{def2.3} \rm
A two parameter function of the Mittag-Lefller type is defined by the series expansion
\begin{eqnarray*}
E_{\alpha,\beta}(y)=\sum_{k=0}^{\infty}{y^k \over\Gamma(\alpha
k+\beta)}=
{1\over2\pi i}\int_c{{\mu}^{\alpha-\beta}e^{\mu}\over{\mu}^{\alpha}-y}d\mu,\;\alpha,\;\beta>0,\;y\in
\mathbb{C},
\end{eqnarray*}
where $c$ is a contour which starts and ends at $-\infty$ and encircles the disc $|\mu|\leq|y|^{1\over\alpha}$ counter clockwise.
The Laplace integral of this function given by
\begin{eqnarray*}
\int_0^{\infty}e^{-\lambda t}t^{\beta-1}E_{\alpha,\beta}(\omega
t^{\alpha})dt=
{{\lambda}^{\alpha-\beta}\over{\lambda}^{\alpha}-\omega},\;
Re\lambda>{\omega}^{{1\over\alpha}},\; \omega>0.
\end{eqnarray*}
Let $A$ be positive definite operator which is linear and closed then Laplace integral of Mittag-Lefller function
\begin{eqnarray*}
\int_0^{\infty}e^{-\lambda t}t^{\beta-1}E_{\alpha,\beta}(A t^{\alpha})dt={{\lambda}^{\alpha-\beta}\over{\lambda}^{\alpha}I-A},\;
Re\lambda>{A}^{{1\over\alpha}}.
\end{eqnarray*}
\end{definition}

\begin{definition}\label{def2.6} \rm
 Let $A:D(A)\subset X\to X$ be a closed and linear operator and $\alpha,\beta>0.$ We say that $A$ is the generator of  $(\alpha,\beta)$ operator function if there exists $\omega \geq 0$ and a strongly continuous
function $W_{\alpha,\beta}:\mathbb{R}^+\rightarrow L(X)$ such that
 $\{{\lambda}^{\alpha}:Re\lambda>\omega\}\subset\rho(A)$ and
$${\lambda}^{\alpha-\beta}({\lambda}^{\alpha}I-A)^{-1}u=\int_0^{\infty}e^{-\lambda
t}W_{\alpha,\beta}(t)udt,\; Re\lambda>\omega,\; u\in X.$$
Here $W_{\alpha,\beta}(t)$ is called the operator function
generated by $A.$
\end{definition}

\begin{remark}\rm
The operator function $W_{\alpha,\beta}(t)$ is general case of $\alpha$-resolvent family and solution operator. In case $\beta=1,$ operator function correspond to solution operator  $S_{\alpha}(t)$ by definition 2.1 in \cite{rpbd}, whereas in the case $\beta=\alpha,$ operator function correspond to $\alpha$-resolvent family defined in \cite{dacl} in definition (2.3) and operator function correspond to $K_{\alpha}(t)$ in \cite{xsqw} if case $\beta=2.$
\end{remark}

\begin{lemma} \rm
Let the function $f$ continuous and $A$ is a sectorial operator of the type $(M,\theta,\alpha,\mu)$. Consider differential equation of order $\alpha\in(1,2)$
\begin{eqnarray}\label{ME}
^CD_t^\alpha u(t)&=&Ay(t)+f(t),\;t\in J=[0,T],t\neq t_k,\\
\label{con}u(0)&=&u_0\in X,\;u'(0)=u_1\in X,\\
\label{ipm}\Delta u(t_k)&=&I_k(u(t_k^-)),\Delta u'(t_k)=Q_k(u(t_k^-)),\;t\neq t_k,\quad k=1,2,...m.
\end{eqnarray}
Then a function $u(t)\in PC^2([0,T],X)$
 is a solution of the system (\ref{ME})-(\ref{ipm})
if it satisfies the following integral equation
 \begin{eqnarray*}
 u(t)=\begin{cases}S_{\alpha }(t)u_0+K_{\alpha }(t)u_1+\int_0^tT_{\alpha }(t-s)f(s)ds,& t\in (0,t_1]\\
 S_{\alpha }(t)u_0+K_{\alpha }(t)u_1+\sum _{i=1}^kS_{\alpha }(t-t_i)I_i(u(t_i^-))\\
+\sum _{i=1}^kK_{\alpha }(t-t_i)Q_i(u(t_i^-))+\int_0^tT_{\alpha }(t-s)f(s)ds,& t\in (t_k,t_{k+1}],\label{p1eq7}
 \end{cases}
 \end{eqnarray*}
where $S_{\alpha}(t),K_{\alpha}(t),T_{\alpha}(t)$ are operators generated by $A$ and defined as
\begin{eqnarray*}
&&S_\alpha(t)=\frac{1}{2\pi i}\int_{\Gamma}e^{\lambda t}\lambda^{\alpha-1}(\lambda^\alpha I-A)^{-1}d\lambda;\;K_\alpha(t)=\frac{1}{2\pi i}\int_{\Gamma}e^{\lambda t}\lambda^{\alpha-2}(\lambda^\alpha I-A)^{-1}d\lambda,\\
&&T_\alpha(t)=\frac{1}{2\pi i}\int_{\Gamma}e^{\lambda
t}(\lambda^\alpha I-A)^{-1}d\lambda
\end{eqnarray*}
with $\Gamma$ is a suitable path such that
$\lambda^{\alpha}\not\in\mu+S_{\theta}$ for $\lambda\in \Gamma.$
\end{lemma}
\begin{proof} If $t\in (0,t_1],$ we have following problem
\begin{eqnarray}
\label{p1eq}&^CD_t^\alpha u(t)=Au(t)+f(t),\\
& \label{icon}u(0)=u_0,\;u'(0)=u_1.
\end{eqnarray}
By Lemma $3.1$ in \cite{jwxl}, the solution of Eq.
(\ref{p1eq})-(\ref{icon}) can be written as
\begin{eqnarray}
\label{int1}u(t)=u_0+u_1t+\int_0^t{(t-s)^{\alpha -1}\over\Gamma (\alpha )}Au(s)ds+\int_0^t{(t-s)^{\alpha -1}\over\Gamma (\alpha )}f(s)ds.
\end{eqnarray}
If $ t\in (t_k,t_{k+1}],k=1,2,...m,$ we have the following equations
\begin{eqnarray}
\label{p1eq1}^CD_t^\alpha u(t)&=&Au(t)+f(t),\\
 u(t_k^+)&=&u(t_k^-)+I_k(u(t_k^-)),\\
\label{icon1} u'(t_k^+)&=&u'(t_k^-)+Q_k(u(t_k^-)).
\end{eqnarray}
By Lemma $3.1$ in \cite{jwxl} the solution of Eq.
(\ref{p1eq1})-(\ref{icon1}),  can be written as
\begin{eqnarray}
\nonumber u(t)&=&u_0+u_1t+\sum _{i=1}^kI_i(u(t_i^-))+\sum _{i=1}^kQ_i(u(t_i^-))(t-t_i)\\
&&+\int_0^t{(t-s)^{\alpha -1}\over\Gamma \alpha} Au(s)ds+\int_0^t{(t-s)^{\alpha -1}\over\Gamma (\alpha )}f(s)ds.\label{p1eq2}
\end{eqnarray}
Summarizing  Eq. (\ref{int1}) and Eq. (\ref{p1eq2}) to $t\in (0,T],$ we get
\begin{eqnarray}
\nonumber u(t)&=&u_0+u_1t+\sum _{i=1}^m\chi_{t_i}(t)I_i(u(t_i^-))+\sum _{i=1}^m\chi_{t_i}(t)Q_i(u(t_i^-))(t-t_i)\\
&&+\int_0^t{(t-s)^{\alpha -1}\over\Gamma \alpha} Au(s)ds+\int_0^t{(t-s)^{\alpha -1}\over\Gamma (\alpha )}f(s)ds,\label{p1eq3}
\end{eqnarray}
where $\chi_{t_i}(t)=\begin{cases}0\qquad t\leq  t_i\\
 1\qquad t>t_i.\end{cases}$

By taking the Laplace transformation on Eq. (\ref{p1eq3}), we have
\begin{eqnarray}
\nonumber L\{u(t)\}&=&{u_0\over\lambda }+{u_1\over\lambda^2 }+\sum _{i=1}^m{e^{-\lambda t_i}\over\lambda }I_i(u(t_i^-))+\sum _{i=1}^m{e^{-\lambda t_i}\over\lambda^2 }Q_i(u(t_i^-))\\
&&+{A\over\lambda ^{\alpha }}L\{u(t)\}+{1\over\lambda ^{\alpha }}L\{f(t)\}.\label{gg}
\end{eqnarray}
On simplifying Eq. (\ref{gg}), we get
\begin{eqnarray} \nonumber L\{u(t)\}&=&{\lambda ^{\alpha -1}(u_0)\over{(\lambda ^{\alpha }I-A)}}+{\lambda ^{\alpha -2}(u_1)\over{(\lambda ^{\alpha }I-A)}}+\sum _{i=1}^m{\lambda ^{\alpha -1}\over{(\lambda ^{\alpha }I-A)}}e^{-\lambda t_i}I_i(u(t_i^-))\\
&&+\sum _{i=1}^m{\lambda ^{\alpha -2}\over{(\lambda ^{\alpha }I-A)}}e^{-\lambda t_i}Q_i(u(t_i^-))+{1\over{(\lambda ^{\alpha }I-A)}}L\{f(t)\}.\label{p1eq5}
\end{eqnarray}
Now, taking the inverse Laplace transformation of Eq. (\ref{p1eq5}), we have
\begin{eqnarray}
\nonumber u(t)&=&S_{\alpha }(t)u_0+K_{\alpha }(t)u_1+\sum _{i=1}^mS_{\alpha }(t-t_i)\chi_{t_i}(t)I_i(u(t_i^-))\\
&&\nonumber+\sum _{i=1}^mK_{\alpha }(t-t_i)\chi_{t_i}(t)Q_i(u(t_i^-))+\int_0^tT_{\alpha }(t-s)f(s)ds,\label{pleq6}\hspace{.75in} t\in J.\label{p1eq7}
\end{eqnarray}
This complete the proof of the lemma.\end{proof}

Now, we state the definition of mild solution of problem (\ref{p})-(\ref{ci}).
\begin{definition}\rm
A function $u:(-\infty,T]\to X$ such that $u\in PC_t^2,
u(0)=\phi(0),u'(0)=\varphi(0),$ is called a mild solution of problem
(\ref{p})-(\ref{ci}) if it satisfies the following integral equation
\begin{eqnarray*}
 u(t)=\begin{cases}S_{\alpha }(t)\phi(0)+K_{\alpha }(t)\varphi(0)+\int_0^tT_{\alpha }(t-s)f(s,u_{\rho(s,u_s)})ds,& t\in (0,t_1]\\
 S_{\alpha }(t)\phi(0)+K_{\alpha }(t)\varphi(0)+\sum _{i=1}^kS_{\alpha }(t-t_i)I_i(u(t_i^-))\\
+\sum _{i=1}^kK_{\alpha }(t-t_i)Q_i(u(t_i^-))+\int_0^tT_{\alpha }(t-s)f(s,u_{\rho(s,u_s)})ds,& t\in (t_k,t_{k+1}].\label{p1eq7}
 \end{cases}
 \end{eqnarray*}
\end{definition}

For the analysis and to prove the results, we use following fixed point theorems.
\begin{theorem}[Banach fixed point theorem]\label{baft}\rm
Let $C$ be a closed subset of a Banach space $X$ and let $f$ be contraction mapping from $C$ in to $C.$ i.e.
$$\|f(y)-f(z)\|\leq \delta\|y-z\|\;\;\forall\;y,z\in C;\;0<\delta<1.$$
Then there exists a unique $z\in C$ such that $f(z)=z.$
\end{theorem}

\begin{theorem}[Krasnoselkii's fixed point theorem] \label{kraski}\rm
Let $B$ be a closed convex and nonempty subset of a Banach space
$X$. Let $P$ and $Q$ be two operators such that
\begin{itemize}
\item[(i)] $Px+Qy\in B$, whenever $x,y\in B$.
\item[(ii)](ii) $P$ is compact and continuous.
\item[(iii)] $Q$ is a contraction mapping.
\end{itemize}
Then there exists $z\in B$ such that $z=Pz+Qz$.
\end{theorem}

\begin{theorem}[Nonlinear Leray-Schauder Alternative]\label{laray} \rm
Let $D$ be a closed convex subset of Banach space $X$ and assume that $0\in D .$ Let $ P:D\rightarrow D$ be a completely continuous mapping. Then, either the set $\{y \in D:y=\lambda P (y), \lambda \in (0,1)\}$ is unbounded or the map $P $ has a fixed point in $D.$
\end{theorem}

\section{Existence Result of mild solution}
In this section, we prove the existence of mild solutions for
(\ref{p})-(\ref{ci}) with a non-convex valued right-hand side. If
$A$ is sectorial operator then the strongly continuous function
$\|W_{\alpha,\beta}(t)\|_{L(X}\leq M.$ We have
$\|S_{\alpha}(t)\|\leq M;\|K_{\alpha}(t)\|\leq
M;\|T_{\alpha}(t)\|\leq M.$ To prove our results,  we shall assume
the function $ \rho $  is continuous. Our result is based on
contraction fixed point theorem for this we have following
assumptions
\begin{enumerate}
\item[$(H_1)$] $f$ is continuous and there exists $l_{f}\in L^1(J,\mathbb{R}^+)$ such that
$$\|(f(t,\psi)-f(t,\xi))\|_X\leq l_f(t)\|\psi-\xi\|_{PC_0}\;\mbox{for every}\; \psi,\xi\in PC_0.$$
\item[$(H_2)$] The functions $I_k,Q_k:X\to X,$ are continuous and
there exists $l_{i},l_{j}\in L^1(J,\mathbb{R}^+)$ such that
$$\|I_k(x)-I_k(y)\|_X\leq l_i(t)\|x-y\|_X;\|Q_k(x)-Q_k(y)\|_X\leq l_j(t)\|x-y\|_X,$$ for all $x,y\in X$ and $k=1,\dots,m.$
\end{enumerate}

\begin{theorem}\rm\label{cnt}
Let the assumption $(H_1)$ and $(H_2)$ hold and the constant
$$\Delta=M\left[m\|l_i\|_{L^1(J,\mathbb{R}^+)}+m\|l_j\|_{L^1(J,\mathbb{R}^+)}+\int^T_0l_f(s)ds\right]<1.$$ Then problem (\ref{p})-(\ref{ci}) has a unique mild solution $u\in PC^2_T$ on $J.$
\end{theorem}
\begin{proof} In order to prove the main result, first we convert the problem (\ref{p})-(\ref{ci}) into fixed point problem. Consider the  Banach space $PC_T^2=\left\{u\in PC_t^2:u(0)=\phi(0),u'(0)=\varphi(0)\right\}$ and defined the operator $P:PC_T^2\to PC_T^2$ by
\begin{eqnarray}
 Pu(t)=\begin{cases}S_{\alpha }(t)\phi(0)+K_{\alpha }(t)\varphi(0)+\int_0^tT_{\alpha }(t-s)f(s,u_{\rho(s,u_s)})ds,& t\in [0,t_1]\\
 S_{\alpha }(t)\phi(0)+K_{\alpha }(t)\varphi(0)+\sum _{i=1}^kS_{\alpha }(t-t_i)I_i(u(t_i^-))\\
+\sum _{i=1}^kK_{\alpha }(t-t_i)Q_i(u(t_i^-))+\int_0^tT_{\alpha }(t-s)f(s,u_{\rho(s,u_s)})ds,& t\in (t_k,t_{k+1}].\label{p1eq7}
 \end{cases}
 \end{eqnarray}
It is clear that $u$ is unique mild solution of the problem (\ref{p})-(\ref{ci}) if and only if u is a solution of the operator
equation $Pu=u.$ Let $u, u^*\in PC_T^2,$ for $[0,t_1]$ we have
\begin{eqnarray*}
\|Pu-Pu^*\|_X&\leq&\int_0^t\|T_{\alpha }(t-s)\|_{L(X}\|f(s,u_{\rho(s,u_s)})-f(s,u_{\rho(s,u_s)})\|_Xds,
\end{eqnarray*}
using  the assumptions $(H_1)$ we get
\begin{eqnarray*}
\|Pu-Pu^*\|_{PC_T^2}&\leq & M\left[\int^T_0l_f(s)ds\right]\|u-u^*\|_{PC_T^2}.
\end{eqnarray*}
Now, without lose of generality we consider the subinterval $(t_k,t_{k+1}]$ to prove our result.
\begin{eqnarray*}
\|Pu-Pu^*\|_X&\leq&\sum _{i=1}^k\|S_{\alpha }(t-t_i)\|_{L(X}\|I_i(u(t_i^-))-I_i(u^*(t_i^-))\|_X\\
&&+\sum _{i=1}^k\|K_{\alpha }(t-t_i)\|_{L(X}\|Q_i(u(t_i^-))-Q_i(u^*(t_i^-))\|_X\\
&&+\int_0^t\|T_{\alpha }(t-s)\|_{L(X}\|f(s,u_{\rho(s,u_s)})-f(s,u_{\rho(s,u_s)})\|_Xds
\end{eqnarray*}
again using  the assumptions $(H_1)$ and $(H_2)$ we get
\begin{eqnarray*}
\|Pu-Pu^*\|_{PC_T^2}&\leq & M\left[m\|l_i\|_{L^1(J,\mathbb{R}^+)}+m\|l_j\|_{L^1(J,\mathbb{R}^+)}+\int^T_0l_f(s)ds\right]\|u-u^*\|_{PC_T^2}\\
&\leq & \Delta\|u-u^*\|_{PC_T^2}.
\end{eqnarray*}
Since $\Delta <1,$ which implies that $P$ is contraction map on Banach space. Hence $P$ has a unique fixed point, which is the mild solution of problem (\ref{p})-(\ref{ci}) on $J.$ This completes the proof of the theorem.
\end{proof}

To apply the Krasnoselkii's  fixed point theorem we have following
assumptions.
\begin{enumerate}
\item[$(H_3)$] There exists a function $m_{f}\in L^1(J,\mathbb{R}^+)$ such that
$$\|(f(t,\psi)-f(t,\xi))\|_X\leq m_f(t)\|\psi\|_{PC_0}\;\mbox{for every}\; \psi\in PC_0.$$
\item[$(H_4)$] The functions $I_k,J_k:X\to X,$ are continuous and
there exist positive constants $C_{i},C_{j},$ such that
$$\|I_k(x)\|_X\leq C_i;\|Q_k(x)\|_X\leq C_j,$$
for all $x\in X$ and $k=1,\dots,m.$
\end{enumerate}

\begin{theorem}\rm\label{Kras}
Let the assumption $(H_1); (H_3);(H_4)$ hold and there exists a constant
$$\Theta=M\int_0^tm_f(s)ds<1.$$ Then problem (\ref{p})-(\ref{ci}) has a mild solution $u\in PC^2_T$ on $J.$
\end{theorem}

\begin{proof}
Let us choose a number $r$ such that
$$r\ge \frac{M[\|\phi(0)\|+\|\varphi(0)\|+mC_i+mC_j]}{1-M\int_0^tm_f(s)ds}$$
Define a set $\mathcal{B}_r=\{u\in PC_T^2:\|u\|_{PC_T^2}\leq r\}$ which is bounded, closed and convex subset $PC_T^2.$ Consider the operators $P_1:\mathcal{B}_r\to \mathcal{B}_r$ and $P_2:\mathcal{B}_r\to \mathcal{B}_r$ for the interval $(t_k,t_{k+1}]$ by
\begin{eqnarray*}
 P_1u(t)&=& S_{\alpha }(t)\phi(0)+K_{\alpha }(t)\varphi(0)+\sum _{i=1}^kS_{\alpha }(t-t_i)I_i(u(t_i^-))
+\sum _{i=1}^kK_{\alpha }(t-t_i)Q_i(u(t_i^-)).\label{p1eq7}\\
 P_2u(t)&=&\int_0^tT_{\alpha }(t-s)f(s,u_{\rho(s,u_s)})ds.\label{p1eq7}
 \end{eqnarray*}
Let $u, u^*\in \mathcal{B}_r,$ we have
\begin{eqnarray*}
\|P_1u(t)+P_2u^*(t)\|_X&\leq& \|S_{\alpha }(t)\|_{L(X}\|\phi(0)\|+\|K_{\alpha }(t)\|_{L(X}\|\varphi(0)\|\\
&&+\sum _{i=1}^k\|S_{\alpha }(t-t_i)\|_{L(X}\|I_i(u(t_i^-))\|\\
&&+\sum _{i=1}^k\|K_{\alpha }(t-t_i)\|_{L(X}\|Q_i(u(t_i^-))\|\\
&&+\int_0^t\|T_{\alpha }(t-s)\|_{L(X}\|f(s,u*_{\rho(s,u^*_s)})\|ds
\end{eqnarray*}
using  the assumptions $(H_1); (H_3)$ and $(H_4)$ we get
\begin{eqnarray*}
\|P_1u(t)+P_2u^*(t)\|_{\mathcal{B}_r}&\leq&
M\left[\|\phi(0)\|+\|\varphi(0)\|+mC_i+mC_j+
r\int_0^tm_f(s)ds\right]\\
&\leq&\frac{M[\|\phi(0)\|+\|\varphi(0)\|+mC_i+mC_j]}{1-M\int_0^tm_f(s)ds}\le r
 \end{eqnarray*}
It is obvious that $\mathcal{B}_r$
is closed with respect to the maps $P_1$ and $P_2.$

Now, we show that $P_1$ is continuous and compact map. Consider a sequence
$\{u^n\}_{n=1}^\infty$  in
 $\mathcal{B}_r$ with $\lim u^n\to u$ in $\mathcal{B}_r.$ Then
\begin{eqnarray*}
\|P_1u^n(t)-P_1u(t)\|_X&\leq&\sum _{i=1}^k\|S_{\alpha }(t-t_i)\|_{L(X)}\|I_i(u^n(t_i^-))-I_i(u(t_i^-))\|_X\\
&&+\sum _{i=1}^k\|K_{\alpha }(t-t_i)\|_{L(X)}\|Q_i(u^n(t_i^-))-Q_i(u(t_i^-))\|_X.
 \end{eqnarray*}
Since the functions $I_i,Q_i\; i=1,2,\dots,m,$ are continuous,
hence $\lim_{n\to\infty}P_1u^n=P_1u$ in $\mathcal{B}_r,$ which show that $P_1$ is continuous map on $\mathcal{B}_r.$

 Let $u\in \mathcal{B}_r,$ we have
\begin{eqnarray*}
\|P_1u(t)\|_X&\leq& \|S_{\alpha }(t)\|_{L(X}\|\phi(0)\|+\|K_{\alpha }(t)\|_{L(X}\|\varphi(0)\|+\sum _{i=1}^k\|S_{\alpha }(t-t_i)\|_{L(X}\|I_i(u(t_i^-))\|\\
&&+\sum _{i=1}^k\|K_{\alpha }(t-t_i)\|_{L(X}\|Q_i(u(t_i^-))\|
\end{eqnarray*}
using  the assumptions $(H_4)$ we get
\begin{eqnarray*}
\|P_1u(t)\|_{\mathcal{B}_r}&\leq&
M\left[\|\phi(0)\|+\|\varphi(0)\|+mC_i+mC_j\right]=\mathcal{C}^*.
 \end{eqnarray*}
It proves that $P_1$ maps bounded set into  bounded sets in $\mathcal{B}_r.$
Consider $P_1(\mathcal{B}_r)=\{P_1u:u\in \mathcal{B}_r\}$ is an equi-continuous family of functions. Next, we show that $P_1$ maps bounded set
 into equi-continuous sets in $P_1(\mathcal{B}_r).$ Let $l_1,l_2 \in (t_i,t_{i+1}],t_i\leq l_1 <l_2 \leq t_{i+1},i=0,1,\dots,m,$ then we have
\begin{eqnarray*}
\|Pu(l_2)-Pu(l_1)\|_X&\leq& \|S_{\alpha }(l_2)-S_{\alpha }(l_1)\|_{L(X}\|\phi(0)\|_{PC_0}+\|K_{\alpha }(l_2)-K_{\alpha }(l_1)\|_{L(X}\|\varphi(0)\|_{PC_0}\\
&&+\sum _{i=1}^k\|S_{\alpha }(l_2-t_i)-S_{\alpha }(l_1-t_i)\|_{L(X}\|I_i(u(t_i^-))\|_X\\
&&+\sum _{i=1}^k\|K_{\alpha }(l_2-t_i)-K_{\alpha }(l_1-t_i)\|_{L(X}\|Q_i(u(t_i^-))\|_X\\
&\leq& \|S_{\alpha }(l_2)-S_{\alpha }(l_1)\|_{L(X}\|\phi(0)\|_{PC_0}+\|K_{\alpha }(l_2)-K_{\alpha }(l_1)\|_{L(X}\|\varphi(0)\|_{PC_0}\\
&&+\sum _{i=1}^k\|S_{\alpha }(l_2-t_i)-S_{\alpha }(l_1-t_i)\|_{L(X}\|I_i(u(t_i^-))\|_X\\
&&+\sum _{i=1}^k\|K_{\alpha }(l_2-t_i)-K_{\alpha }(l_1-t_i)\|_{L(X}\|Q_i(u(t_i^-))\|_X
 \end{eqnarray*}
Since $K_{\alpha }(t),S_{\alpha }(t)$ are strongly continuous functions so
 $K_{\alpha }(l_2)-K_{\alpha }(l_1);S_{\alpha }(l_2)-S_{\alpha }(l_1)\to 0$ as $l_1\to l_2,$ which implies that $\|P_1u(l_2)-P_1u(l_1)\|_X\to 0$ as $l_1\to l_2.$ Hence $P(\mathcal{B}_r)$ is equi-continuous.
 We conclude that the operator $P_1$ is compact by Arzela Ascoli's theorem.

 Let $u, u^*\in \mathcal{B}_r,$ we have the following step
 \begin{eqnarray*}
 \|P_2u(t)-P_2u^*(t)\|&\le &\int_0^tT_{\alpha }(t-s)\|f(s,u_{\rho(s,u_s)})-f(s,u^*_{\rho(s,u^*_s)})\|ds\\
 &\le &M\int_0^tm_f(s)ds\|u-u^*\|\\
 &= & \Theta\|u-u^*\|
 \end{eqnarray*}
 Since $\Theta < 1,$ it implies that $P_2$ is a contraction map.
 By Krasnoselkii's theorem {} set $\mathcal{B}_r$ has a fixed
point $u(t)$ which is the solution of system (\ref{p})-(\ref{ci}). This completes the
proof of the theorem.
 \end{proof}

To use Nonlinear Lerey-Schauder Alternative fixed point theorem we have following assumptions.
\begin{enumerate}
\item[$(H_3)$] The functions $I_k,J_k:X\to X,$ are continuous and
there exists positive constants $C_{i},C_{j},$ such that
$$\|I_k(x)\|_X\leq C_i;\|Q_k(x)\|_X\leq C_j,$$ for all $x\in X$ and $k=1,\dots,m.$
\item[$(H_4)$] There exist $m_f\in C(J, [0,\infty))$ and a continuous non-decreasing function $\Omega_f : [0,\infty)\to(0,\infty)$ such that
$$\|f(t,\varphi)\|\leq m_f(t)\Omega_f (\|\varphi\|),\;\forall \;(t,\varphi)\in J \times PC_0.$$
\end{enumerate}

\begin{theorem}\rm\label{leray}
Let the assumption $(H_3)$ and $(H_4)$ hold and
\begin{eqnarray*}
M\int_0^T m_f(s)ds<\int^{\infty}_{C'}\frac{1}{\Omega_f (s)}ds,
 \end{eqnarray*}
where $C'=M[\|\phi(0)\|+\|\varphi(0)\|+mC_i+mC_j],$ then there exists a mild solution of problem (\ref{p})-(\ref{ci}) on $J.$
\end{theorem}
\begin{proof}
Consider the operator $P:PC_T^2\to PC_T^2$ defined as in Theorem
\ref{cnt}. Now, we show that $P$ is completely continuous map for
general subinterval $(t_k,t_{k+1}]$. For this purpose firstly, let
$\{u^n\}_{n=1}^\infty$ be a sequence in
 $PC_T^2$ with $\lim u^n\to u$ in $PC_T^2.$
\begin{eqnarray*}
\|Pu^n(t)-Pu(t)\|_X&\leq&\sum _{i=1}^k\|S_{\alpha }(t-t_i)\|_{L(X)}\|I_i(u^n(t_i^-))-I_i(u(t_i^-))\|_X\\
&&+\sum _{i=1}^k\|K_{\alpha }(t-t_i)\|_{L(X)}\|Q_i(u^n(t_i^-))-Q_i(u(t_i^-))\|_X\\
&&+\int_0^t\|T_{\alpha }(t-s)\|_{L(X)}\|f(s,u^n_{\rho(s,u^n_s)})-f(s,u_{\rho(s,u_s)})\|_Xds.
 \end{eqnarray*}
Since the functions $f,I_i,Q_i\; i=1,2,\dots,m,$ are continuous,
hence $\lim_{n\to\infty}Pu^n=Pu$ in $PC_T^2,$ which implies that the
map $P$ is continuous on $PC_T^2.$ Consider set
$\mathcal{B}_r=\{u\in P:PC_T^2:\|u\|_{PC_T^2}\leq r\}.$ It clear
that $\mathcal{B}_r$  is a bounded, closed convex subset in
$PC_T^2.$ Let $u\in \mathcal{B}_r,$ we have
\begin{eqnarray*}
\|Pu(t)\|_X&\leq& \|S_{\alpha }(t)\|_{L(X}\|\phi(0)\|+\|K_{\alpha }(t)\|_{L(X}\|\varphi(0)\|+\sum _{i=1}^k\|S_{\alpha }(t-t_i)\|_{L(X}\|I_i(u(t_i^-))\|\\
&&+\sum _{i=1}^k\|K_{\alpha }(t-t_i)\|_{L(X}\|Q_i(u(t_i^-))\|+\int_0^t\|T_{\alpha }(t-s)\|_{L(X}\|f(s,u_{\rho(s,u_s)})\|ds
\end{eqnarray*}
using  the assumptions $(H_3)$ and $(H_4)$ we get
\begin{eqnarray*}
\|Pu(t)\|_{\mathcal{B}_r}&\leq&
M\left[\|\phi(0)\|+\|\varphi(0)\|+mC_i+mC_j+\Omega_f
(r)\int_0^tm_f(s)ds\right]=C^*.
 \end{eqnarray*}
It proves that $P$ maps bounded set into  bounded sets in $\mathcal{B}_r$  for all sub interval $t \in (t_i,t_{i+1}],i=1,2,..,m.$
Consider $P(\mathcal{B}_r)=\{Pu:u\in \mathcal{B}_r\}$ is an equi-continuous family of functions. Next, we show that $P$ maps bounded set
 into equi-continuous sets in $P(\mathcal{B}_r).$ Let $l_1,l_2 \in (t_i,t_{i+1}],t_i\leq l_1 <l_2 \leq t_{i+1},i=0,1,\cdots,m,$ then we have
\begin{eqnarray*}
\|Pu(l_2)-Pu(l_1)\|_X&\leq& \|S_{\alpha }(l_2)-S_{\alpha }(l_1)\|_{L(X}\|\phi(0)\|_{PC_0}+\|K_{\alpha }(l_2)-K_{\alpha }(l_1)\|_{L(X}\|\varphi(0)\|_{PC_0}\\
&&+\sum _{i=1}^k\|S_{\alpha }(l_2-t_i)-S_{\alpha }(l_1-t_i)\|_{L(X}\|I_i(u(t_i^-))\|_X\\
&&+\sum _{i=1}^k\|K_{\alpha }(l_2-t_i)-K_{\alpha }(l_1-t_i)\|_{L(X}\|Q_i(u(t_i^-))\|_X\\
&&+\int_0^t\|T_{\alpha }(l_2-s)-T_{\alpha }(l_1-s)\|_{L(X}\|f(s,u_{\rho(s,u_s)})\|_Xds\\
&&+\int_{l_1}^{l_2}\|T_{\alpha }(l_2-s)\|_{L(X}\|f(s,u_{\rho(s,u_s)})\|_Xds\\
&\leq& \|S_{\alpha }(l_2)-S_{\alpha }(l_1)\|_{L(X}\|\phi(0)\|_{PC_0}+\|K_{\alpha }(l_2)-K_{\alpha }(l_1)\|_{L(X}\|\varphi(0)\|_{PC_0}\\
&&+\sum _{i=1}^k\|S_{\alpha }(l_2-t_i)-S_{\alpha }(l_1-t_i)\|_{L(X}\|I_i(u(t_i^-))\|_X\\
&&+\sum _{i=1}^k\|K_{\alpha }(l_2-t_i)-K_{\alpha }(l_1-t_i)\|_{L(X}\|Q_i(u(t_i^-))\|_X\\
&&+\int_0^t\|T_{\alpha }(l_2-s)-T_{\alpha }(l_1-s)\|_{L(X}\|f(s,u_{\rho(s,u_s)})\|_Xds\\
&&+(l_2-l_1)M\|f(s,u_{\rho(s,u_s)})\|_X.
 \end{eqnarray*}
Since $T_{\alpha }(t),K_{\alpha }(t),S_{\alpha }(t)$ are strongly continuous functions so
 $T_{\alpha }(l_2)-T_{\alpha }(l_1);K_{\alpha }(l_2)-K_{\alpha }(l_1);S_{\alpha }(l_2)-S_{\alpha }(l_1)\to 0$ as $l_1\to l_2,$ which implies that $\|Pu(l_2)-Pu(l_1)\|_X\to 0$ as $l_1\to l_2.$ Hence $P(\mathcal{B}_r)$ is equi-continuous.

Finally, we establish  a priori estimate for the solutions of the
integral equation $u = \lambda Pu$ for $\lambda\in (0, 1).$ Let $u$
be a solution of $z = \lambda Pz,\lambda\in (0, 1),$ then we have
\begin{eqnarray*}
\|u(t)\|_X&\leq& \|S_{\alpha }(t)\|_{L(X}\|\phi(0)\|_{PC_0}+\|K_{\alpha }(t)\|_{L(X}\|\varphi(0)\|_{PC_0}+\sum _{i=1}^k\|S_{\alpha }(t-t_i)\|_{L(X}\|I_i(u(t_i^-))\|_X\\
&&+\sum _{i=1}^k\|K_{\alpha }(t-t_i)\|_{L(X}\|Q_i(u(t_i^-))\|+\int_0^t\|T_{\alpha }(t-s)\|_{L(X}\|f(s,u_{\rho(s,u_s)})\|_Xds\\
&&\leq M\left[\|\phi(0)\|_{PC_0}+\|\varphi(0)\|_{PC_0}+mC_i+mC_j+\int_0^tm_f(s)\Omega_f (\|u_{\rho(s,u_s)}\|)ds\right]
\end{eqnarray*}
again using  the assumptions $(H_1)$ and $(H_2)$ we get
\begin{eqnarray*}
\|u(t)\|_{PC^2_T}&\leq& M\left[\|\phi(0)\|_{PC_0}+\|\varphi(0)\|_{PC_0}+mC_i+mC_j+\int_0^tm_f(s)\Omega_f (\|u\|_{PC^2_T})ds\right].
 \end{eqnarray*}
If $\beta_{\lambda}(t)=\|u\|_{PC^2_T},$ we get that
\begin{eqnarray*}
\beta_{\lambda}(t)&\leq&M\left[\|\phi(0)\|_{PC_0}+\|\varphi(0)\|_{PC_0}+mC_i+mC_j+\int_0^tm_f(s)\Omega_f (\beta_{\lambda}(t))ds\right].
 \end{eqnarray*}
Then, we get
\begin{eqnarray*}
\beta_{\lambda}'(t)&\leq&Mm_f(t)\Omega_f (\beta_{\lambda}(t)),
 \end{eqnarray*}
and hence
\begin{eqnarray*}
\int^{\beta_{\lambda}(t)}_{C'=\beta(0)}\frac{1}{\Omega_f (s)}ds\leq M\int_0^T m_f(s)ds.
 \end{eqnarray*}
It is clear that set of functions $\{\beta_{\lambda}:\lambda \in
(0,1)\}$ is bounded, which implies that $\{u:\lambda \in (0,1)\}$ is
bounded. By nonlinear alternative of Leray-Schauder fixed theorem,
we deduce that $P$ has a fixed point $u,$ which is a mild solution
of the problem (\ref{p})-(\ref{ci}) on $J.$
\end{proof}

\section{Application}
Consider the following impulsive fractional partial differential equation of the form
\begin{eqnarray}
\label{ex1}{\partial^\alpha \over\partial
t^\alpha}u(t,x)&=&{\partial^2\over\partial y^2}u(t,x)+\frac{u(t-\sigma_1(t)\sigma_2(\|u\|),x)}{49},t\neq\frac{1}{2},\\
u(t,0)&=&u(t,\pi)=0;u'(t,0)=u'(t,\pi)=0\;\;t\geq 0, \nonumber\\
u(t,x)&=&\phi(t,x),u'(t,x)=\varphi(t,x),t\in[-d,0],x\in[0,\pi],\nonumber\\
\label{ex2}\Delta
u|_{t=\frac{1}{2}}&=&\frac{\|u(\frac{1^-}{2})\|}{25+\|u(\frac{1^-}{2})\|},\;\Delta
u'|_{t=\frac{1}{2}}=\frac{\|u(\frac{1^-}{2})\|}{16+\|u(\frac{1^-}{2})\|},
\nonumber
\end{eqnarray}
where ${\partial^\alpha\over\partial t^\alpha}$ is Caputo's
fractional derivative of order $\alpha\in(1,2),0<t_1<t_2<\cdots<t_n<T$ are prefixed numbers and $\phi,\varphi\in PC_0.$
Let $X=L^2[0,\pi]$ and define the operator
$A:D(A)\subset X\to X$ by $Aw=w''$ with the domain $D(A):=\{w\in X:
\;w,w'$ are absolutely continuous, $w''\in X, w(0)=0=w(\pi)\}.$ Then
$$Aw=\sum_{n=1}^{\infty} n^2(w,w_n)w_n,\;w\in D(A),$$ where
$w_n(x)=\sqrt{2\over\pi}\sin (nx),\; n\in\mathbb{N}$ is the
orthogonal set of eigenvectors of $A$.
It is well known that A is the infinitesimal
generator of an analytic semigroup $(T (t))_{t \geq 0}$ in $X$ and is given by
$$T (t)\omega =\sum_{n=1}^\infty e^{-n^2t  }(\omega,\omega_n)\omega_n,\; \mbox{for all}\; \omega \in X,\; \mbox{and every}\; t > 0.$$
 We assume that $\rho_i:[0,\infty)\to[0,\infty),\;i=1,2,$ are
continuous functions.

Set $u(t)(x)=u(t,x),$ and $\rho(t,\phi)=\rho_1(t)\rho_2(\|\phi(0)\|)$ we have
\begin{eqnarray*}
f(t,\phi)(x)=\frac{\phi}{49},I_{k}(u)=\frac{\|u\|}{25+\|u\|},J_{k}(u)=\frac{\|u\|}{16+\|u\|},
\end{eqnarray*}
then with these settings the problem (\ref{ex1})-(\ref{ex2}) can be
written in the abstract form of equation (\ref{p})-(\ref{ci}). It is
obvious that the maps $f,I_k,J_k$ follow the assumption $ H_1,H_2$,
this implies that there exists a unique mild solution of problem
(\ref{ex1})-(\ref{ex2}) on $[0,T].$

\end{document}